\documentclass[11pt]{article}

\usepackage{geometry}
\usepackage{color}
\usepackage{float}
\usepackage{textcomp}
\usepackage{amsthm}
\usepackage{amsmath}
\usepackage{amssymb}

\usepackage{changes}

\usepackage{mathtools}

\newtheorem{theorem}{Theorem}[section]
\newtheorem{lemma}[theorem]{Lemma}

\newtheorem{proposition}[theorem]{Proposition}

\newtheorem{definition}[theorem]{Definition}

\DeclareMathOperator{\argmin}{argmin}

\DeclareMathOperator{\dom}{dom}

\newcommand{\R}{\mathbb{R}}

\newcommand{\Dom}{\mathrm{Dom}}

\newcommand{\inner}[2]{\langle{#1},{#2}\rangle}

\newcommand{\norm}[1]{\|#1\|}

\newcommand{\ri}{{\mbox{\rm ri\,}}}
\newcommand{\tos}{\rightrightarrows} 
\newcommand{\Y}{\mathcal{Y}}
\newcommand{\X}{\mathcal{X}}
\newcommand{\Z}{\mathcal{Z}}

\newcommand{\V}{\mathcal{V}}

\newcommand{\Sf}{\mathcal{S}}
\newcommand{\inte}{\mathrm{int}}

\newcommand{\vgap}{\vspace{.1in}}

\newcommand{\tx}{\tilde x}

\newcommand{\tz}{\tilde z}

\newcommand{\bi}{\begin{itemize}}
\newcommand{\ei}{\end{itemize}}
\newcommand{\ba}{\begin{array}}
\newcommand{\ea}{\end{array}}

\begin{document}

\title{Extending the ergodic convergence rate of the proximal ADMM}

\author{
    Max L.N. Gon\c calves
    \thanks{Institute of Mathematics and Statistics, Federal University of Goias, Campus II- Caixa
    Postal 131, CEP 74001-970, Goi\^ania-GO, Brazil. (E-mails: {\tt
       maxlng@ufg.br} and {\tt jefferson@ufg.br}).  The work of these authors was
    supported in part by  CNPq Grant  444134/2014-0, 309370/2014-0  and FAPEG/GO.}
    \and
      Jefferson G. Melo \footnotemark[1]
    \and
    Renato D.C. Monteiro
    \thanks{School of Industrial and Systems
    Engineering, Georgia Institute of
    Technology, Atlanta, GA, 30332-0205.
    (email: {\tt monteiro@isye.gatech.edu}). The work of this author
    was partially supported by NSF Grant CMMI-1300221.}
}

\date{November 8, 2016}

\maketitle

\begin{abstract}
Pointwise and ergodic iteration-complexity results for the proximal alternating direction method of multipliers (ADMM) for any
stepsize in $(0,(1+\sqrt{5})/2)$ have been recently established in the literature.
In addition to giving alternative proofs of these results, this paper also extends the ergodic iteration-complexity result to include the
case in which the stepsize is equal to $(1+\sqrt{5})/2$. As far as we know, this is the first ergodic iteration-complexity for
the stepsize $(1+\sqrt{5})/2$  obtained in the ADMM literature.
These results are obtained by showing that the proximal ADMM is 
an instance of a non-Euclidean hybrid proximal extragradient framework whose pointwise and ergodic convergence rate are also
studied.
 \end{abstract}

%
%
%
%
%
%
%
%
%
%
%
%
%
%
%
%
%
%
%
%
%
%
%
%
%

\pagestyle{plain}

\section{Introduction} \label{sec:int}
This paper considers  the following  linearly constrained convex problem 
\begin{equation} \label{optl}
\inf \{ f(y) + g(s) : C y + D s = c \}
\end{equation}
where  $\Sf$, $\Y$ and $\X$ are  finite dimensional inner product spaces, 
$f: \Y \to (-\infty,\infty]$ and $g:\Sf \to (-\infty,\infty]$ are proper
closed convex functions,  $C: \Y \to \X$ and $D: \Sf \to \X$ are linear operators,
and $c \in \X$.  Convex optimization problems with a separable structure such as \eqref{optl}
 appear in many applications areas such as machine learning, compressive sensing and image processing. A well-known method that takes
 advantage of the special structure of \eqref{optl}
is the  alternating direction method of multipliers (ADMM).

 Many variants of the ADMM have been considered  in the literature; see, for example, \cite{ChPo_ref2,Deng1,PADMM_Eckstein,GADMM2015,MJR,Gu2015,Hager,HeLinear,Lin,LanADMM}. 
Here, we  study the proximal ADMM \cite{PADMM_Eckstein,FPST_editor} which, recursively, computes a sequence
$\{(s_k,y_k,x_k)\}$ as follows. 
Given $(s_{k-1},y_{k-1},x_{k-1})$, the $k$-th triple $(s_k,y_k,x_k)$ is determined as
\begin{align}
s_k &= \argmin_{s} \left \{ g(s) - \inner{ {x}_{k-1}}{Ds} +
\frac{\beta}{2} \| C y_{k-1} + D s - c \|^2 + \frac12 \inner{ s-s_{k-1}}{H(s-s_{k-1})}\right\},\nonumber\\
y_k &= \argmin_y \left \{ f(y) - \inner{x_{k-1}}{Cy} +
\frac{\beta}{2} \| C y + D s_k - c \|^2+\frac12 \inner{ y-y_{k-1}}{G(y-y_{k-1})}\ \right\},\label{ADMMclass}\\
x_k &= x_{k-1}-\theta\beta\left[Cy_k+Ds_k-c\right] \nonumber
\end{align}
where $\beta > 0$ is a penalty parameter, $\theta>0$ is a  stepsize parameter, and
 $H: \mathcal{S} \to \mathcal{S}$ and  $G: \Y \to \Y$ are positive semidefinite self-adjoint linear operators.
We refer to the subclass obtained from  \eqref{ADMMclass} by setting $(H,G)=(0,0)$ to as the standard ADMM.
Also, the proximal ADMM with $(H,G)=(\tau I-\beta D^{*}D$,0) for some $\tau \ge \beta \|D\|^2$  is known as the
linearized ADMM or the split inexact Uzawa method (see, e.g.,  \cite{HeLinear,Xahang,Zhang2010}).
It has the desirable feature that, for many applications, its subproblems are   much easier to solve or  even have  closed-form solutions  (see   \cite{Deng1,HeLinear,Wang2012,Yang_linearizedaugmented} for more details).

Pointwise and ergodic iteration-complexity results for the proximal ADMM \eqref{ADMMclass} for any $\theta \in (0,(1+\sqrt{5})/2)$ are established in \cite{Cui,Gu2015}. Our paper develops alternative pointwise and ergodic iteration-complexity results for the proximal ADMM \eqref{ADMMclass}
based on a different but related termination criterion. More specifically, a pointwise iteration-complexity is established for any $\theta \in (0,(1+\sqrt{5})/2)$
and an ergodic one is obtained for any $\theta \in (0,(1+\sqrt{5})/2]$. Hence, our analysis of the ergodic case includes the case
$\theta=(1+\sqrt{5})/2$ which, as far as we know, has not been established yet.
Our approach towards obtaining this extension is based on viewing the proximal ADMM as an instance of a non-Euclidean hybrid proximal extragradient (HPE) framework
whose (pointwise and ergodic) complexity is studied and is then used to derive that of the proximal ADMM. 
\\[2mm]
{\bf Previous related works.} 
 The ADMM   was  introduced in \cite{0352.65034,0368.65053} and is thoroughly discussed in \cite{Boyd:2011,glowinski1984}.
To discuss complexity results about ADMM, we use the terminology
weak pointwise or strong pointwise bounds
to refer to complexity bounds relative to the best of the $k$ first iterates or the last iterate, respectively, to
satisfy a suitable termination criterion.
The first iteration-complexity bound for the  ADMM  was established only recently in  
\cite{monteiro2010iteration} under the assumptions that $C$ is injective.
More specifically, the ergodic iteration-complexity for the standard ADMM  is derived in  \cite{monteiro2010iteration}   for any $\theta \in (0,1]$ while
a weak pointwise iteration-complexity easily follows from the approach in \cite{monteiro2010iteration} for any $\theta \in (0,1)$.
Subsequently, without assuming that $C$ is injective, \cite{HeLinear} established the ergodic iteration-complexity of the proximal  ADMM  \eqref{ADMMclass}
with $G=0$ and $\theta=1$ and, as a consequence, of the  split inexact Uzawa method~\cite{Xahang}.
Paper \cite{He2} establishes the weak pointwise and ergodic iteration-complexity  
of another collection of ADMM instances which includes the standard ADMM for any $\theta \in (0,(1+\sqrt{5})/2)$.
A strong pointwise iteration-complexity bound for the proximal ADMM   \eqref{ADMMclass} with $G=0$ and $\theta =1$ is derived in~\cite{He2015}.
Finally, a number of papers (see for example \cite{Cui,Deng1,GADMM2015,MJR,Gu2015,Hager,Lin,LanADMM} and  references therein) have extended
most of these complexity results to the context of the ADMM class \eqref{ADMMclass} as well as other ADMM classes.

The non-Euclidean HPE framework is a class of inexact proximal point methods
for solving the monotone inclusion problem  which uses  a  relative (instead of summable) error criterion.
The proximal point method, proposed by Rockafellar \cite{Rock:ppa}, is a classical iterative scheme for solving the latter problem.
Paper \cite{Sol-Sv:hy.ext} introduces an Euclidean version of the HPE framework. 
Iteration-complexities of the latter framework are established in \cite{monteiro2010complexity} (see also \cite{monteiro2011complexity}).
Generalizations of the HPE framework to the non-Euclidean setting are studied in \cite{MJR,Oliver,Sol-Sv:hy.breg}.
Applications of the HPE framework can be found for example in \cite{YHe2,YHe1,Maicon,monteiro2010complexity,monteiro2011complexity,monteiro2010iteration}.
 \\[2mm]
{\bf Organization of the paper.}  Subsection~\ref{sec:bas}  presents our notation and basic results. Section~\ref{subsec:Admm1} describes the  proximal ADMM and present its pointwise and ergodic convergence rate results whose proofs are given in Section \ref{sec:proximal ADMM_proof}. Section~\ref{sec:smhpe} is devoted to the study of a non-Euclidean  HPE framework. This section is divided into two subsections, Subsection~\ref{sec:NE-HPE} introduces the framework and presents  its convergence rate bounds whose proofs are given in Subsection~\ref{sec:hpe_Analysis}.

\subsection{Notation and basic results}
\label{sec:bas}

This subsection presents some definitions, notation and basic results used in this paper.

Let  $\V$ be a finite-dimensional 
real vector space with inner product and associated norm denoted by $\inner{\cdot}{\cdot}_\V$ and $\|\cdot\|_{\V}$, respectively.
For a given  self-adjoint positive semidefinite linear operator $A:\V\to \V$, 
 the seminorm induced by $A$ on $\V$ is defined by $\|\cdot\|_{\V,A}= \langle A (\cdot), \cdot\rangle_{\V} ^{1/2}$.
For an arbitrary seminorm $\|\cdot\|$ on $\V$,  its dual (extended) seminorm, denoted by $\|\cdot\|^*$, is defined as
 $\|\cdot\|^*:=\sup\{ \inner{\cdot}{v}_{\V}:\|v\|\leq 1\}$.  
 
 
 The following result gives some properties of $ \|\cdot\|_{\V,A}^*$  whose proof is omitted.

\begin{proposition}\label{propdualnorm} 
Let $A:\V\to \V$ be a self-adjoint positive semidefinite linear operator. Then,  $\dom  \|\cdot\|_{\V,A}^*=\mbox{Im}\;(A)$ and
$\| A(\cdot)\|_{\V,A}^*=\|\cdot\|_{\V,A}$.
\end{proposition}

Given a set-valued operator $T:\V\tos \V$,
its domain and graph  are defined as  
 $\mbox{Dom}\,T:=\{v\in \V\,:\, T(v)\neq \emptyset\}$ and 
$
Gr(T)=\{ (v_1,v_2)\in \V\times \V\;|\; v_2 \in T(v_1)\},
$
respectively, and  its {inverse} operator $T^{-1}:\V\tos \V$ is given by
$T^{-1}(v_2):=\{v_1\;:\; v_2\in T(v_1)\}$.
The operator $T$ is said to be   monotone if 
\[
\inner{u_1-v_1}{u_2-v_2}\geq 0\quad \forall \;  (u_1,u_2), (v_1,v_2) \,\in \, Gr(T).
\]
Moreover, $T$ is maximal monotone if it is monotone and there is no other monotone operator $S$ such that $Gr(T)\subset Gr(S)$.
Given a scalar $\varepsilon\geq0$, the 
 {$\varepsilon$-enlargement} $T^{[\varepsilon]}:\V\tos \V$
 of a monotone operator $T:\V\tos \V$ is defined as
\begin{align}
\label{eq:def.eps}
 T^{[\varepsilon]}(v)
 :=\{v'\in \V\,:\,\inner{v'-v_2}{v-v_1}\geq -\varepsilon,\;\;\forall (v_1,v_2)\in Gr(T)\} \quad \forall v \in \V.
\end{align}

Recall that the 
{$\varepsilon$-subdifferential} of a 
 convex function $f:\V\to  [-\infty,\infty]$
is defined by
$\partial_{\varepsilon}f(v):=\{u\in \V\,:\,f(v')\geq f(v)+\inner{u}{v'-v}-\varepsilon\;\;\forall v'\in \V\}$ for every $v\in \V$.
When $\varepsilon=0$, then $\partial_0 f(x)$ 
is denoted by $\partial f(x)$
and is called the {subdifferential} of $f$ at $x$.
The operator $\partial f$ is trivially monotone if $f$ is proper.
If $f$ is a proper lower semi-continuous convex function, then
$\partial f$ is maximal monotone~\cite{Rockafellar}.
The domain of $f$ is denoted by $\dom f$ and 
the conjugate of $f$ is the function
$f^*:\V \to [-\infty,\infty]$ defined as
\[
f^*(v) = \sup_{z \in \V} \left(\inner{v}{z} - f(z) \right)\quad \forall v \in \V.
\]
\section{Proximal ADMM and its  convergence rate}\label{subsec:Admm1}

In this section, we recall the proximal ADMM for solving  \eqref{optl} and present pointwise and ergodic convergence rate results.  The pointwise convergence rate considers the stepsize parameter in the open interval  $(0,(\sqrt{5}+1)/2)$ while the ergodic one  includes also the stepsize $(\sqrt{5}+1)/2$.

Throughout this section, we assume that:
\begin{itemize}
  \item[\bf A1)] the
  problem \eqref{optl} has an optimal solution $(s^*,y^*)$ and an associated Lagrange multiplier $x^*$, or equivalently,
   the inclusion
\begin{equation} \label{FAB}
0\in T(s,y,x) := \left[ \begin{array}{c}  \partial g(s)- D^*{x}\\ \partial f(y)- C^* {x} \\ Cy+Ds-c
\end{array} \right] 
\end{equation}
has a solution $(s^*,y^*,x^*)$;
  \item[\bf A2)]  there exists  $x\in \X$
  such that $(C^*x,D^*x) \in \ri(\dom f^*) \times \ri(\dom g^*)$.
\end{itemize}

Next we  state the   proximal ADMM  for solving  the problem \eqref{optl}.
\\
\hrule
\noindent
\\
{\bf Proximal ADMM}
\\
\hrule
\begin{itemize}
\item[(0)] Let an initial point $(s_0,y_0,x_0) \in \mathcal{S}\times \Y\times  \X$, a penalty parameter $\beta>0$, a setpsize $\theta>0$, and
self-adjoint positive semidefinite linear operators  $H: \mathcal{S} \to \mathcal{S}$ and  $G: \Y \to \Y$  be given, and set $k=1$;
\item[(1)]    compute an optimal solution $s_k \in \Sf$  of the subproblem
\begin{equation} \label{def:tsk-admm}
\min_{s \in \mathcal{S}} \left \{ g(s) - \inner{ D^*{x}_{k-1}}{s}_\mathcal{S} +
\frac{\beta}{2} \| C y_{k-1} + D s - c \|^2_\X+\frac{1}{2}\|s- s_{k-1}\|_{\mathcal{S},H}^2 \right\}
\end{equation}
and compute an optimal solution $y_k\in \Y$ of the subproblem
\begin{equation} \label{def:tyk-admm}
\min_{y \in \Y} \left \{ f(y) - \inner{ C^*{x}_{k-1}}{y}_\Y +
\frac{\beta}{2} \| C y + D s_k - c \|^2_\X +\frac{1}{2}\|y- y_{k-1}\|_{\Y,G}^2\right\};
\end{equation}

\item[(2)] set 
\begin{equation}\label{admm:eqxk}
x_k = x_{k-1}-\theta\beta\left[Cy_k+Ds_k-c\right]
\end{equation}
and $k \leftarrow k+1$, and go to step~(1).
\end{itemize}
{\bf end}\\
\hrule
\noindent
\\ 

The proximal ADMM  has different features depending on the choice of the operators $H$ and $G$. 
 For instance, by taking $(H,G)=(0,0)$ and $(H,G)=(\tau I-\beta D^{*}D$,0) with  $\tau>0$,
it reduces to the standard ADMM and the linearized ADMM, respectively.
The latter method   is related to  the split inexact Uzawa method (see, e.g.,  \cite{HeLinear,Zhang2010}) and 
it basically consists of linearizing the quadratic term $(1/2)\| C y_{k-1} + D s - c \|^2_\X$ in the standard ADMM and adding a proximal term $(1/2)\|s-s_{k-1}\|^2_{\mathcal{S},H}$. In many applications,  the  corresponding subproblem~\eqref{def:tsk-admm} for the linearized ADMM is  much easier to solve or  even has a closed-form solution (see   \cite{HeLinear,Wang2012,Yang_linearizedaugmented} for more details). 
We also mention that depending on the structure of problem~\eqref{optl}, other choices of  $H$ and $G$ may be recommended; see, for instance,  Section~1.1 of \cite{Deng1}.
It is worth pointing out that the condition {\bf A2} is used only to  ensure that the subproblems of ADMM as well as some variants  have solutions,
see for example \cite[Proposition~7.2]{monteiro2010iteration} and \cite[comments on page~16]{MJR}. In particular, under this assumption it is possible to show that the subproblems \eqref{def:tsk-admm} and \eqref{def:tyk-admm} have solutions.

The next two results present pointwise and ergodic convergence rate bounds for the proximal ADMM
under the assumption that $\theta\in (0,(\sqrt{5}+1)/2)$ and $\theta\in (0,(\sqrt{5}+1)/2]$, respectively.
Their statements use the quantities $d_0$, $\tau_\theta$ and $\sigma_\theta$ defined as

\begin{align}
d_0&:=d_0(\beta,\theta)=\inf_{(s,y,x) \in T^{-1}(0)} \left\{\frac{1}{2}\|s_0-s\|_{\mathcal{S},H}^2+\frac{1}{2}\|y_0-y\|_{\Y,(G+\beta C^*C)}^2+\frac{1}{2\beta\theta}\|x_0-x\|_\X^2\right\}\label{def:d0admm},\\[2mm]
 \sigma_\theta&:=\frac{3\theta^2-7\theta+5+\sqrt{(3\theta^2-7\theta+5)^2-4(2-\theta)(3-\theta)(\theta-1)^2}}{2(3-\theta)},\label{def:sigma}\\[2mm]
 \tau_\theta&:=4\max\left\{\frac{1}{\sqrt{\theta}},\frac{\sqrt{\theta}}{2-\theta}\right\}\label{def:tau}.
\end{align}

It is easy to verify that $\sigma_\theta \in (0,1)$ whenever $\theta \in  (0,(\sqrt{5}+1)/2)$ and
$\sigma_\theta=1$ when $\theta=(\sqrt{5}+1)/2$.

\begin{theorem} ({\bf Pointwise convergence of the proximal ADMM}) \label{th:pointwise} 
Consider the sequence $\{(s_k,y_k,x_k)\}$  generated by the proximal ADMM  with $\theta\in (0,(\sqrt{5}+1)/2)$,
and let $\{\tx_k\}$  be defined as 
\begin{equation}\label{xtilde}
\tilde{x}_{k}={x}_{k-1}-\beta(Cy_{k-1}+Ds_k-c).
\end{equation}
Then, for every $k\in\mathbb{N}$, 
\begin{equation}\label{eq:th_incADMMtheta1} 
\left( 
\begin{array}{c} 
H (s_{k-1}-s_k)\\[1mm]  
(G+\beta C^*C)(y_{k-1}-y_k)\\[1mm]  
\frac{1}{\beta\theta}(x_{k-1}-x_k)
\end{array} 
\right) \in 
\left[ 
\begin{array}{c} 
\partial g(s_k)- D^*\tilde{x}_k\\[1mm]  
\partial f(y_k)- C^*\tilde{x}_k\\[1mm]  
Cy_k+Ds_k-c
\end{array} \right]
\end{equation}
and   there exists $i\leq k$ such that
\[
  \left(\|s_{i-1}-s_i\|_{\mathcal{S},H}^2+\|y_{i-1}-y_i\|_{\Y,(G+\beta C^*C)}^2+\frac{1}{\beta\theta}\norm{x_{i-1}-x_i}_\mathcal{X}^2\right)^{1/2}\leq \frac{2\sqrt{d_0}}{\sqrt{k}} \sqrt{\frac{1+\sigma_\theta +2\tau_\theta }{1-\sigma_\theta}}
\]
where  $d_0,$ $ \sigma_\theta$ and  $\tau_\theta$  are as in  \eqref{def:d0admm},   \eqref{def:sigma} and \eqref{def:tau}, respectively. 
\end{theorem}

In contrast to the pointwise convergence rate result stated above, the ergodic convergence rate result stated below holds for the extreme case in
which $\theta=(\sqrt{5}+1)/2$.

\begin{theorem} {\bf (Ergodic convergence of the proximal ADMM)}\label{th:ergodicproximal ADMM}
Consider the sequence $\{(s_k,y_k,x_k)\}$  generated by the proximal ADMM with  $\theta\in (0,(\sqrt{5}+1)/2]$, and let $\{\tx_k\}$  be  as in  \eqref{xtilde}.
Moreover, consider the ergodic sequences
$\{(s^a_k,y^a_k,x^a_k,\tx^a_k)\}$ and 
$\{\varepsilon_{k}^a\}$ defined by
\begin{equation}\label{eq:jase12}
 (s_k^a,y_k^a,x_k^a,\tx_k^a):=\frac1k\sum_{i=1}^k\left( s_i, y_i,x_i ,\tx_i \right), \quad
(\varepsilon^a_{k,s},\varepsilon^a_{k,y})= \frac{1}{{k}}\sum_{i=1}^k\left(\inner{r_{i,s}}{s_i-s_k^a}, \inner{r_{i,y}}{y_i-y_k^a}\right). 
\end{equation}
where 
\begin{equation}\label{eq:291} 
(r_{i,s},r_{i,y})=\left(H(s_{i-1}-s_{i}),(G+\beta C^*C)(y_{i-1}-y_{i})\right).
\end{equation}
Then, for every $k\in\mathbb{N}$,
\begin{equation}\label{eq:th_incADMMtheta<1} 
\left( 
\begin{array}{c} 
H (s_{k-1}^a-s_k^a)\\[1mm]  
(G+\beta C^*C)(y_{k-1}^a-y_k^a)\\[1mm]  
\frac{1}{\beta\theta}(x_{k-1}^a-x_k^a)
\end{array} 
\right) \in 
\left[ 
\begin{array}{c} 
\partial g_{\varepsilon^a_{k,s}}(s_k^a)- D^*\tilde{x}_k^a\\[1mm]  
\partial f_{\varepsilon^a_{k,y}}(y_k^a)- C^*\tilde{x}_k^a\\[1mm]  
Cy_k^a+Ds_k^a-c
\end{array} \right]
\end{equation}  
\[
  \left(\|s_{k-1}^a-s_k^a\|_{\mathcal{S},H}^2+\|y_{k-1}^a-y_k^a\|_{\Y,(G+\beta C^*C)}^2+\frac{1}{\beta\theta}\norm{x_{k-1}^a-x_k^a}_\mathcal{X}^2\right)^{1/2}\le \frac{2\sqrt{2(1+\tau_\theta) d_0}}{k}
    \]
 and
 \[
 \varepsilon^a_{k,s} + \varepsilon^a_{k,y} \leq \frac{ 3(1+\tau_\theta)[  3\theta^2+4\sigma_\theta(\theta^2+\theta+1)] d_{0} }{\theta^2k}
\]
where   $d_0,$  $ \sigma_\theta$ and $\tau_\theta$ are as in  \eqref{def:d0admm}, \eqref{def:sigma} and \eqref{def:tau}, respectively. 
\end{theorem}

The proofs of Theorems~\ref{th:pointwise} and \ref{th:ergodicproximal ADMM}  will be presented  in Section~\ref{sec:proximal ADMM_proof}.
For this, we first study a non-Euclidean HPE framework  from which the proximal ADMM is a special instance.

\section
{A non-Euclidean HPE framework}
\label{sec:smhpe}
This section describes and derives convergence rate bounds for a non-Euclidean HPE framework for solving monotone inclusion problems.
Subsection~\ref{sec:NE-HPE} describes the non-Euclidean HPE framework and
 its corresponding pointwise and ergodic convergence rate bounds.  Subsection~\ref{sec:hpe_Analysis} gives the
proofs for the two convergence rate results stated in Subsection~\ref{sec:NE-HPE}.

\subsection{A non-Euclidean HPE framework and its convergence rate}\label{sec:NE-HPE}

Let  $\Z$ be finite-dimensional inner product real vector space.
We start by introducing the definition of a distance generating function and its corresponding
Bregman distance adopted in this paper.

\begin{definition}\label{defw0}
A proper lower semi-continuous convex function $w : \Z  \to (-\infty,\infty]$ is called a distance generating function
if   $\inte(\dom w) = \Dom\, \partial w \neq \emptyset$ and $w$ is continuously differentiable on this interior.
Moreover, $w$ induces the Bregman distance $dw: \Z \times \inte(\dom w) \to \mathbb{R}$ defined as
\begin{equation}\label{def_d}
(dw)(z';z) := w(z')-w(z)-\langle \nabla w(z),z'-z\rangle_{\Z}  \quad \forall (z', z) \in \Z \times \inte(\dom w).
\end{equation}
\end{definition}
For simplicity, for every $z \in \inte(\dom w)$,  the function $(dw)(\,\cdot\,;z)$ will be denoted by $(dw)_{z}$ so that
\[
(dw)_{z}(z')=(dw)(z';z) \quad \forall (z',z) \in \Z \times \inte(\dom w).
\]
The following useful identities follow straightforwardly from \eqref{def_d}:
\begin{align}
\nabla (dw)_{z}(z') &= - \nabla (dw)_{z'}(z) = \nabla w(z') - \nabla w(z) \quad \forall z, z' \in \inte(\dom w), \label{grad-d} \\
\label{equacao_taylor}
(dw)_{v}(z') - (dw)_{v}(z) &= \langle \nabla (dw)_{v}(z), z'-z\rangle_{\Z} + (dw)_{z}(z') \quad \forall z' \in \Z, \, \forall v  , z \in \inte(\dom w).
\end{align}

Our analysis of the non-Euclidean HPE framework  requires  the distance generating function to be regular
with respect to a seminorm according to the following definition.

\begin{definition}\label{def:assu}
Let distance generating function $w:\Z \to [-\infty,\infty]$, seminorm $\|\cdot\|$ in $\Z$ and convex set $Z \subset \inte(\dom w)$
be given.
For given positive constants $m$ and $M$,   $w$ is said to be $(m,M)$-regular with respect to $(Z,\|\cdot\|)$ if 
\begin{align}\label{strongly}
 (dw)_{z}(z')\geq \frac{{m}}{2}\|z-z'\|^2 \quad \forall  z,z' \in  Z,
\end{align}
\begin{equation}\label{a1}
\|\nabla w(z)-\nabla w(z')\|^*\leq {M}\|z-z'\| \quad \forall z, z' \in   Z.
\end{equation}
\end{definition}

We now make some remarks about the class of regular distance generating functions as in Definition \ref{def:assu},
which was first introduced in \cite{MJR}. First, if the seminorm in Definition \ref{def:assu}
is a norm, then \eqref{strongly}  implies that
$w$ is strongly convex, in which case the corresponding $dw$ is said to be nondegenerate.
However, since  $\|\cdot\|$ is not assumed to be a norm, a regular distance generating function $w$ does not
need to be strongly convex, or equivalently, $dw$ can be degenerate. Second, some examples of $(m,M)$-regular distance generating
functions can be found in \cite[Example 2.3]{MJR}. For the purpose of analyzing the proximal ADMM, we
make use of the distance generating function given by $w(\cdot)=(1/2)\|\cdot\|_{\Z,Q}^2$ where $Q$ is a self-adjoint positive semidefinite
linear operator. This $w$ can be easily shown to be $(1,1)$-regular  with respect to $(\Z,\|\cdot\|_{\Z,Q})$.
Third, if  $w:\Z \to [-\infty,\infty]$ is  $(m,M)$-regular with respect to $(Z,\|\cdot\|)$,  then
\begin{equation}\label{lipsc}
\frac{m}{2}\|z-z'\|^2\leq (d{w})_{z}(z')  \leq \frac{M}{2}\|z-z'\|^2  \quad \forall z,z' \in Z.
\end{equation}

Throughout this section, we assume that $w:\Z \to [-\infty,\infty]$ is an $(m,M)$-regular distance generating function 
with respect to $(Z,\|\cdot\|)$ where $Z \subset \inte(\dom\, w)$ is a closed convex set and 
$\|\cdot\|$ is a seminorm in~$\Z$.
Our problem of interest in this section is the monotone inclusion problem (MIP)
\begin{align}\label{eq:inc.p}
 0\in T(z)
\end{align}
where 
$T:\Z\tos \Z$ is a maximal monotone operator and the following conditions hold:
\begin{itemize}
\item[\bf B1)] $\Dom\,T\subset Z$;
\item[\bf B2)]  the solution set $T^{-1}(0)$ of~\eqref{eq:inc.p} is nonempty. 
\end{itemize}

We now state a non-Euclidean HPE (NE-HPE) framework  for solving \eqref{eq:inc.p}.

\vgap
\vgap

\noindent
\fbox{
\begin{minipage}[h]{6.4 in}
{\bf NE-HPE framework for solving \eqref{eq:inc.p}}.
\begin{itemize}
\item[(0)] Let $z_0 \in Z$, $\eta_0 \in \R_{+}$ and  $\sigma \in [0, 1]$ be given, and set $k=1$;
\item[(1)] choose $\lambda_k>0$ and find $(\tilde{z}_k, z_k, \varepsilon_k,\eta_k) \in Z \times Z \times \mathbb{R}_{+}\times \mathbb{R}_{+}$   such that 
       \begin{align}
& r_k:= \frac{1}{\lambda_k} \nabla (dw)_{z_k}(z_{k-1})  \in T^{[\varepsilon_k]}(\tz_k), \label{breg-subpro} \\      
& (dw)_{z_k}({\tz}_k) + \lambda_k\varepsilon_k +\eta_k\leq \sigma (dw)_{z_{k-1}}({\tz}_k)+\eta_{k-1}; \label{breg-cond1}
\end{align}
       
\item[(2)] set $k\leftarrow k+1$ and go to step 1.
\end{itemize}
\noindent
{\bf end}
\end{minipage}
}

\vgap
\vgap

We now make some remarks about the  NE-HPE framework. First, \cite{MJR} studies an NE-HPE framework based on a regular distance
generating function $w$ for solving a monotone inclusion problem consisting of the sum of $T$ and a $\mu$-monotone operator $S$
with respect to $w$. The latter notion implies strong monotonicity of $S$ when $dw$ is nondegenerate (see \cite[Assumption (A1)]{MJR}).
Second, the NE-HPE does not specify
how to find $\lambda_k$ and $(\tilde{z}_k, z_k, \varepsilon_k)$ satisfying (\ref{breg-subpro}) and (\ref{breg-cond1}).
The particular 
scheme for computing $\lambda_k$ and $(\tilde{z}_k, z_k, \varepsilon_k)$ will depend on the instance of the framework under consideration
and the properties of the operator $T$.
Third, if $w$ is strongly convex on $Z$ and $\sigma= 0$, then (\ref{breg-cond1}) implies that $\varepsilon_k= 0$
and $z_k = \tilde z_k$ for every~$k$, and
hence that $r_k \in T(z_k)$ in view of \eqref{breg-subpro}.
Therefore, the HPE error conditions \eqref{breg-subpro}-\eqref{breg-cond1} can be viewed as a relaxation of an iteration of the exact non-Euclidean proximal point method,
namely,
\[
0 \in \frac{1}{\lambda_k} \nabla (dw)_{z_{k-1}}(z_{k})  +  T({z}_k).
\]
Fourth, if $w$ is strongly convex on $Z$, then it can be shown that the above inclusion has a unique solution $z_k$,
and hence that, for any given $\lambda_k>0$, there exists a triple
 $(\tz_k,z_k,\varepsilon_k)$ of the form $(z_k,z_k,0)$ satisfying
(\ref{breg-subpro})-(\ref{breg-cond1}) with $\sigma=0$. Clearly, computing the triple
 in this (exact) manner is expensive, and hence computation of (inexact) quadruples satisfying the HPE (relative) error conditions with $\sigma>0$
is more computationally appealing.

We end this subsection by presenting pointwise and ergodic convergence rate results for the NE-HPE framework whose
proofs are given in the next subsection.
Their statements use the quantity $(dw)_0$ defined as 
\begin{equation}\label{d0HPE}
(dw)_0 = \inf \{ (dw)_{z_0}(z^*) : z^* \in T^{-1}(0)\}.
\end{equation}

\begin{theorem} ({\bf Pointwise convergence of the NE-HPE}) \label{th:pointwiseHPE}
Consider the sequence $\{(r_k,\varepsilon_k,\lambda_k)\}$  generated by the NE-HPE framework with $\sigma <1 $.
 Then, for every $k \ge 1$, $r_k \in T^{[\varepsilon_k]}(\tz_k)$ and  the following statements hold:
  \begin{itemize}
\item[(a)]  if $\underline{\lambda}:=\inf_{j \ge 1} \lambda_j>0$,
then there exists $i\leq k$ such that
  \[
  \norm{r_i}^* \leq  \frac{2M}{\sqrt{m}} 
    \sqrt{\frac{(1+\sigma)(dw)_0+2\eta_0}{1-\sigma}
    \, \left( \frac{\lambda_i^{-1}}
   { \sum_{j=1}^k \lambda_j}\right)}
\leq
\frac{2M}{\underline\lambda\sqrt{m k}} \sqrt{\frac{(1+\sigma)(dw)_0+2\eta_0}{1-\sigma}}
\]
\[
  \varepsilon_i\leq
\frac{(1+\sigma) (dw)_0+2\eta_0}{(1-\sigma)\sum_{i=1}^k\lambda_i}
\leq
\frac{(1+\sigma) (dw)_0+2\eta_0}{(1-\sigma)\underline
    \lambda k};
  \]
  \item[(b)] there exists an index $i\leq k$ such that
  \begin{equation*}
  \norm{r_i}^*\leq \frac{2M}{\sqrt{m}} 
    \sqrt{\frac{(1+\sigma)(dw)_0+2\eta_0}{1-\sigma}
    \, \left( \frac{1}
   { \sum_{j=1}^k \lambda_j^2}\right)},
  \quad \quad \quad
  \varepsilon_i\leq\frac{[(1+\sigma)(dw)_0+2\eta_0]\lambda_i}{(1-\sigma)\sum_{j=1}^k \lambda_j^2},
\end{equation*}
\end{itemize}
where $(dw)_0$ is as defined in \eqref{d0HPE}.
\end{theorem}

From now on, we focus on the ergodic convergence of the NE-HPE framework.
For $k \geq 1$, define $\Lambda_{k} := \sum_{i=1}^k \lambda_i$ and the ergodic iterate $(\tz^a_k,r^a_k,\varepsilon^a_{k})$ as
\begin{equation}\label{SeqErg}
\tilde z^a_{k} = \frac{1}{\Lambda_{k}}\sum_{i=1}^k \lambda_i \tilde z_i, \quad
r^a_{k} := \frac{1}{\Lambda_{k}}\sum_{i=1}^k \lambda_i r_i, \quad
\varepsilon^a_{k} := \frac{1}{\Lambda_{k}} \sum_{i=1}^k \lambda_i \left( \varepsilon_i + \inner{r_i}{\tilde z_i -\tilde z^a_{k}} \right).
\end{equation}
 
 The following result provides  convergence rate bounds for  $\|r^a_k\|^*$ and $\varepsilon_k^a$. The pair $(r^a_k,\varepsilon_k^a)$ plays the role of a residual for $\tz^a_k$. 

\begin{theorem} {\bf (Ergodic convergence of the NE-HPE)}\label{th:ergHPE}
For every $k\geq 1$, $r^a_k \in T^{[\varepsilon_k^a]}(\tz^a_k)$ and
\begin{align*}
 \|r_k^a\|^* &\le \frac{2\sqrt{2}M((dw_0)+\eta_0)^{1/2}}{\sqrt{m}\Lambda_k}, \quad
 \varepsilon^a_{k} \leq \left(\frac{3M}{m}  \right)
\left[\frac{ 3 ((dw)_{0} +\eta_0) + \sigma \rho_k}{\Lambda_k}\right].
\end{align*}
where
\[
\rho_k := \displaystyle\max_{i=1,\ldots,k}(dw)_{z_{i-1}}(\tilde z_{i}).
\]
Moreover, the sequence $\{\rho_k\}$ is bounded under either
one of the following situations:
\begin{itemize}
\item[(a)]
$\sigma<1$, in which case
\begin{equation} \label{def:tauk}
\rho_k \le \frac{(dw)_0+\eta_0}{1-\sigma};
\end{equation}
\item[(b)]
$\Dom\, T$ is bounded, in which case
\[
\rho_k \le \frac{2M}{m} [ (dw)_0 +\eta_0+ D]  \label{def:tauk1}
\]
where $D := \sup \{ \min\{ (dw)_y(y'), (dw)_{y'}(y) \} :  y,y' \in \Dom \,T\}$ is the diameter
of $\Dom \,T$ with respect to $dw$, and $(dw)_0$ is as defined in \eqref{d0HPE}.

\end{itemize}
\end{theorem}

The bound on  $\varepsilon^a_k$  presented in Theorem~\ref{th:ergHPE} depends on the quantity $\rho_k$ which is bounded under the assumption $\sigma<1$ or boundedness of  $\Dom \,T$. As we will show in Section~\ref{sec:proximal ADMM_proof}, proximal ADMM is an instance of the NE-HPE in which the stepsize $\theta=(\sqrt{5}+1)/2$ corresponds to the parameter $\sigma=1$. Even in this case,  the sequence $\{\rho_k\}$ is bounded regardless the boundedness  of $\Dom \,T$.
\subsection{Convergence rate analysis of the NE-HPE framework}\label{sec:hpe_Analysis}
The main goal of this subsection is to present the proofs of Theorems~\ref{th:pointwiseHPE} and \ref{th:ergHPE}. 
Toward this goal,  we first  establish some technical lemmas which provide useful properties of regular Bregman distances and  of the NE-HPE framework.
\begin{lemma}\label{basicassu}
Let $w: \Z \to [-\infty,\infty]$ be an $(m,M)$-regular distance generating function with respect to $(Z,\|\cdot\|)$ as in Definition \ref{def:assu}.
Then, the following statements hold:
\begin{itemize}
\item[(a)] for every $ z,z' \in Z$, we have
\begin{equation}\label{eq:789}
 \left(\|\nabla(dw)_{z'}(z)\|^*\right)^2\leq  \frac{2 M^2}{m} \min \{ (dw)_{z}(z') ,  (dw)_{z'}(z) \};
\end{equation}
\item[(b)] for every $l \ge 1$ and  $u_0,u_1,\ldots,u_l \in Z$, we have
\begin{equation}\label{eq:56}
(dw)_{u_0}(u_l) \le \frac{lM}{m} \sum_{i=1}^l \min \{ (dw)_{u_{i-1}}(u_i) ,  (dw)_{u_i}(u_{i-1}) \}.
\end{equation}
\end{itemize}
\end{lemma}
\proof
 (a) It is easy to see that \eqref{eq:789} immediately follows from
\eqref{grad-d},  \eqref{strongly} and \eqref{a1}. 

(b) It follows from the second inequality in \eqref{lipsc} that 
\begin{align*}
(dw)_{u_0}(u_l) &\le \frac{M}2 \|u_l-u_0\|^2 \le \frac{M}2 \left ( \sum_{i=1}^l \|u_i - u_{i-1}\| \right)^2
\le \frac{lM}2 \sum_{i=1}^l \|u_i - u_{i-1}\|^2 
\end{align*}
which, in view of \eqref{strongly}, immediately implies \eqref{eq:56}. 
\endproof

The next  result presents some useful estimates related to the sequence generated by the NE-HPE framework.

\begin{lemma}\label{lema_desigualdades}
For every $k \geq 1$, the following statements hold:
\begin{itemize}
\item[(a)] for every $z \in \dom\, w$, we have
\[(dw)_{z_{k-1}}(z) - (dw)_{z_k} (z) = (dw)_{z_{k-1}}(\tilde{z}_k) - (dw)_{z_k} (\tilde{z}_k) + \lambda_k \langle r_k, \tilde{z}_k - z \rangle_{\Z};\]
\item[(b)]
 for every $z \in \dom\, w$, we have
\[(dw)_{z_{k-1}}(z) - (dw)_{z_k}(z) +\eta_{k-1}\ge (1 - \sigma)(dw)_{z_{k-1}}(\tilde{z}_k) + \lambda_k (\langle r_k, \tilde{z}_k - z\rangle_{\Z} + \varepsilon_k)+\eta_k;\]
\item[(c)] for every  $z^* \in T^{-1}(0)$, we have
\[(dw)_{z_{k-1}}(z^*) - (dw)_{z_k}(z^*) +\eta_{k-1} \ge (1 - \sigma)(dw)_{z_{k-1}}(\tilde{z}_k)+\eta_k;
\]
\item[(d)] for every  $z^* \in T^{-1}(0)$, we have
\[(dw)_{z_k}(z^*) +\eta_k \leq (dw)_{z_{k-1}}(z^*) +\eta_{k-1}.
\]
\end{itemize}
\end{lemma}
\proof 
(a) Using (\ref{equacao_taylor}) twice and using the definition of $r_k$ given by (\ref{breg-subpro}), we obtain
\begin{align*}
 (dw)_{z_{k-1}}(z) - (dw)_{z_k}(z) &= (dw)_{z_{k-1}}(z_k) + \langle \nabla (dw)_{z_{k-1}}(z_k), z - z_k\rangle_{\Z} \\
& = (dw)_{z_{k-1}}(z_k) + \langle \nabla (dw)_{z_{k-1}}(z_k), \tilde{z}_k - z_k\rangle_{\Z} + \langle \nabla (dw)_{z_{k-1}}(z_k), z - \tilde{z}_k\rangle_{\Z} \\
& = (dw)_{z_{k-1}}(\tilde{z}_k) - (dw)_{z_k}(\tilde{z}_k) + \langle \nabla (dw)_{z_{k-1}}(z_k), z - \tilde{z}_k\rangle_{\Z} \\
& = (dw)_{z_{k-1}}(\tilde{z}_k) - (dw)_{z_k}(\tilde{z}_k) + \lambda_k\langle r_k, \tilde{z}_k - z\rangle_{\Z}.
\end{align*}

(b) This statement follows as an immediate consequence of (a) and (\ref{breg-cond1}).

(c) This statement follows from (b), the fact that $0 \in T(z^*)$ and $r_k \in T^{[\varepsilon_k]}(\tz_k)$,
and  \eqref{eq:def.eps}.

(d)  This statement follows as an immediate consequence of (c) and $\sigma\leq 1$.

\endproof

The pointwise convergence rate bounds for the NE-HPE framework will follow directly from the next result which  estimates the residual pair $(r_i,\varepsilon_i)$. \begin{lemma}
  \label{lm:breg.alpha}
  Let $\{(r_k,\varepsilon_k,\lambda_k)\}$ and $(\eta_0,\sigma)$ be given by the NE-HPE framework and assume that $\sigma<1$. Then, for every
  $t\in \R$ and every $k \ge 1$, there exists an $i\leq k$ such that
  \begin{equation*}
    \label{v_ieps_i-bound-a}
    \norm{r_i}^* \leq \frac{2M}{\sqrt{m}} 
    \sqrt{\frac{(1+\sigma)(dw)_0+2\eta_0}{1-\sigma}
    \, \left( \frac{\lambda_i^{t-2}}
   { \sum_{j=1}^k \lambda_j^t}\right)} ,
    \quad \quad \quad
    \varepsilon_i\leq 
     \frac{ (1+\sigma) (dw)_0+2\eta_0} {1-\sigma}
    \left( \frac{\lambda_i^{t-1}}{\sum_{j=1}^k \lambda_j^t}
    \right)
\end{equation*}
where $(dw)_0$ is as defined in \eqref{d0HPE}.
\end{lemma}

\begin{proof}
For every $i \ge 1$, define
\begin{equation*}\label{eq:def_teta_i}
\theta_i = \max \left \{ \frac{m\lambda_i^{2} \left(\|r_i\|^*\right) ^2}{4M^2},\lambda_i \varepsilon_i \right\}.
\end{equation*}
It is easy to see that the conclusion of the lemma will follow if we show that, for every $i \ge 1$, we have
\[
(1-\sigma) \sum_{i=1}^k \theta_i \le (1+\sigma)(dw)_0+2\eta_0.
\]
In order to show that the last inequality hold,  we have, from \eqref{grad-d} and \eqref{breg-subpro}, for every $i \ge 1$
\begin{align*}
\lambda_i \|r_i\|^* &= \|\nabla (dw)_{z_{i-1}}(\tz_i) - \nabla (dw)_{z_{i}}(\tz_i) \|^*
\le \|\nabla (dw)_{z_{i-1}}(\tz_i)\|^* + \|\nabla (dw)_{z_{i}}(\tz_i) \|^* \\
&\le {\frac{\sqrt{2}M}{\sqrt{m}}}\left[ (dw)_{z_{i-1}}(\tz_i)^{1/2} + (dw)_{z_{i}}(\tz_i)^{1/2} \right] \\
& \le {\frac{\sqrt{2}M}{\sqrt{m}}}\left[ (dw)_{z_{i-1}}(\tz_i)^{1/2} + (\sigma(dw)_{z_{i-1}}(\tz_i)+\eta_{i-1}-\eta_i)^{1/2} \right], 
\end{align*}
where the second and  third  inequalities are   due to \eqref{eq:789} and  \eqref{breg-cond1}, respectively. Hence,
$$
\frac{m\lambda_i^{2} \left(\|r_i\|^*\right)^2}{2M^2} \leq 2(1+\sigma)(dw)_{z_{i-1}}(\tz_i) + 2(\eta_{i-1}-\eta_i). 
$$
The previous estimative together with \eqref{breg-cond1} and definition of $\theta_i$ imply that
$\theta_i \leq (1+\sigma)(dw)_{z_{i-1}}(\tz_i)+(\eta_{i-1}-\eta_i)$ for every $ i \ge 1$.
Thus, if $z^* \in T^{-1}(0)$, it follows from Lemma \ref{lema_desigualdades}(c) that
\[
(1-\sigma) \sum_{i=1}^k \theta_i \le (1+\sigma)[(dw)_{z_0}(z^*) - (dw)_{z_k}(z^*) +\eta_0-\eta_k] +(1-\sigma)(\eta_0-\eta_k)
\le (1+\sigma)(dw)_{z_0}(z^*)+2\eta_0.
\]
The desired inequality follows from the latter inequality and the definition of $(dw)_0$ in \eqref{d0HPE}. As a consequence, we obtain the conclusion of the lemma.
\end{proof}
Now we are ready to prove Theorems~\ref{th:pointwiseHPE} and \ref{th:ergHPE} stated in Subsection~\ref{sec:NE-HPE}.
\\[1mm]
 {\bf Proof of Theorem~\ref{th:pointwiseHPE}:}
The inclusion $r_k \in T^{[\varepsilon_k]} (\tilde{z}_k)$ holds due to \eqref{breg-subpro}.  Statements (a) and (b)  follow directly from Lemma~\ref{lm:breg.alpha} with t = 1 and t = 2, respectively.  \hfill{ $\square$}
\\[1mm]
 {\bf Proof of Theorem~\ref{th:ergHPE}:}
The inclusion $r^a_k \in T^{[\varepsilon_k^a]}(\tz^a_k)$ follows from the transportation formula (see  \cite[Theorem~2.3]{Bu-Sag-Sv:teps1}). Now, let $z^* \in T^{-1}(0)$.
Using  \eqref{grad-d}, \eqref{breg-subpro} and \eqref{SeqErg}, we easily see that
\[
\Lambda_k r^a_k = \nabla (dw)_{z_0}(z_k) = \nabla (dw)_{z_0}(z^*) - \nabla (dw)_{z_k}(z^*)
\]
which, together with \eqref{eq:789} and Lemma \ref{lema_desigualdades}(d), imply that
\begin{align*}
\Lambda_k \norm{r^a_k}^* &\leq  \norm{\nabla (dw)_{z_0}(z^*)}^* + \norm{ \nabla (dw)_{z_k}(z^*)}^* \\
&\le \frac{\sqrt{2}M}{\sqrt{m}} [ (dw)_{z_0}(z^*)^{1/2} + (dw)_{z_k}(z^*)^{1/2} ] \le \frac{2\sqrt{2}M}{\sqrt{m}}((dw)_{z_0}(z^*)+\eta_0)^{1/2}.
\end{align*}
This inequality together with definition of $(dw)_0$ clearly imply the bound on $\|r_k^a\|^*$.
To show the bound on $\varepsilon_k^a$, first note that Lemma~\ref{lema_desigualdades}(b) implies
that  for every $z \in W$,
\[(dw)_{z_{0}}(z) +\eta_0\geq 
\sum_{i=1}^k\lambda_i (\langle r_i, \tilde{z}_i - z\rangle_{\Z} + \varepsilon_i).\]
Letting $z=\tilde z^a_{k}$ in the last inequality and using the fact that $(dw)_{z_0}(\cdot)$ is convex,  we obtain
\begin{equation}\label{eq:estim_rhok}
 \max_{i=1,\ldots,k} (dw)_{z_0}(\tz_i) +\eta_0\ge (dw)_{z_{0}}(\tilde z^a_{k})+\eta_0
\geq \sum_{i=1}^k\lambda_i (\langle r_i, \tilde{z}_i - \tz^a_k\rangle_{\Z} + \varepsilon_i) = \Lambda_k \varepsilon_k^a
\end{equation}
where the  equality is due to \eqref{SeqErg}.
On the other hand,  \eqref{eq:56}   implies that, for every $i \ge 1$  and $z^* \in T^{-1}(0)$,
\begin{align*}
(dw)_{z_{0}}(\tilde z_{i}) &\le 
\frac{3M}{m} \left[ (dw)_{z_{i}}(\tilde z_{i})+(dw)_{z_{i}}(z^*)+(dw)_{z_{0}}( z^*) \right] \\
& \le \frac{3M}{m} \left[ \sigma (dw)_{z_{i-1}}(\tilde z_{i})+\eta_{i-1}+(dw)_{z_{i-1}}(z^*)+\eta_{i-1}+(dw)_{z_{0}}( z^*)\right] \\
& \le \frac{3M}{m} \left[ \sigma (dw)_{z_{i-1}}(\tilde z_{i})+2((dw)_{z_{i-1}}(z^*)+\eta_{i-1})+(dw)_{z_{0}}( z^*)\right] \\
& \le \frac{3M}{m} \left[ \sigma (dw)_{z_{i-1}}(\tilde z_{i})+3 (dw)_{z_{0}}( z^*) +2\eta_0\right] 
\end{align*}
where the second inequality is due to  \eqref{eq:56} and   Lemma \ref{lema_desigualdades}(d), and   the   last inequality is due to Lemma~\ref{lema_desigualdades}(d).
Combining the above  relations with \eqref{eq:estim_rhok} and using the definitions of $\rho_k$ and $(dw)_0$ and the fact that $M/m\geq 1$, we conclude that
the bound on $\varepsilon_k^a$ holds.

We now establish the bounds on $\rho_k$ under either one of the conditions  (a) or (b).
First, if $\sigma<1$, then it follows
from Lemma \ref{lema_desigualdades}(c)-(d) that
$(1-\sigma) (dw)_{z_{i-1}} (\tz_i) \le (dw)_{z_{i-1}}(z^*)+\eta_{k-1} \le (dw)_{z_{0}}(z^*) +\eta_0$ for every
$i \ge 1$ and $z^* \in T^{-1}(0)$, and hence that \eqref{def:tauk} holds.
Assume now that $\Dom\, T$ is bounded. Then, it follows from inequality \eqref{eq:56} and Lemma \ref{lema_desigualdades}(d) that, for every $i \ge 1$ and $z^* \in T^{-1}(0)$, 
\[
(dw)_{z_{i-1}} (\tz_i) \le \frac{2M}{m} \left[ (dw)_{z_{i-1}}(z^*) + \min\{ (dw)_{\tz_i}(z^*) , (dw)_{z^*}(\tz_i) \} \right]
\le \frac{2M}{m} \left[ (dw)_{z_{0}}(z^*) +\eta_0+ D \right]
\]
which, in view of definitions of $\rho_k$ and $(dw)_0$,  proves (b).   \hfill{ $\square$}


 \section{Convergence rate analysis of the proximal ADMM}\label{sec:proximal ADMM_proof}

Our goal in this section is to show that the  proximal ADMM is an instance  of the NE-HPE framework for solving the inclusion problem \eqref{FAB} and, as a by-product, establish its  pointwise and ergodic convergence rate bounds presented in Section~\ref{subsec:Admm1}.

We start by presenting a preliminary technical result about the proximal ADMM.

\begin{lemma} \label{pr:aux}
Consider the triple $(s_k,y_k,x_k)$ generated at  the k-iteration of the proximal ADMM and the point
$\tilde{x}_k$ defined in \eqref{xtilde}.
Then,
\begin{eqnarray}
0&\in& H(s_k-s_{k-1})+\left[ \partial g(s_k)-D^*\tilde{x}_k\right],  \label{aux.0}\\[3mm]
0&\in&  (G+\beta C^*C)(y_k-y_{k-1})+\left[\partial f(y_k)-C^*\tx_k\right],\label{aux.2}\\[3mm]
0&=&  \frac{1}{\theta\beta}(x_k-x_{k-1})+\left[Cy_k+Ds_k-c \right],\label{aux.1}\\[3mm]
\tx_k-x_{k-1}&=&\beta C(y_{k}-y_{k-1}) +\frac{x_k-x_{k-1}}{\theta}. \label{aux.3}
\end{eqnarray}
\end{lemma}

\begin{proof}
From the optimality condition of  \eqref{def:tsk-admm}, we have
\begin{equation*}
0 \in \partial g(s_k)- D^*({x}_{k-1}-\beta(Cy_{k-1}+Ds_k-c))+H(s_k-s_{k-1}),
\end{equation*}
which, combined with definition of $\tx_k$ in \eqref{xtilde}, yields  \eqref{aux.0}.
Now, from the optimality condition of \eqref{def:tyk-admm} and  definition of $\tx_k$ in \eqref{xtilde}, we obtain
\begin{align*}
0 &\in \partial f(y_k)-C^*x_{k-1}+\beta C^*(Cy_k+Ds_k-c)+G(y_k- y_{k-1})\\[2mm]
   &= \partial f(y_k)-C^*[x_{k-1}+\beta (C{y}_{k-1}+Ds_k-c)]+\beta C^*C(y_{k}-{y}_{k-1})+G(y_k- y_{k-1}) \\[2mm]
   & = \partial f(y_k)-C^*\tilde{x}_k+\beta C^*C(y_{k}-{y}_{k-1})+G(y_k-y_{k-1}).
\end{align*}
which proves  \eqref{aux.2}.
Moreover, \eqref{aux.1} follows immediately from \eqref{admm:eqxk}.
On the other hand, 
it follows from  definition of   ${x}_k$ in  \eqref{admm:eqxk} that
\begin{align*}
\frac{x_k-x_{k-1}}{\theta}+\beta C(y_k-y_{k-1})=  -\beta(Cy_{k-1}+Ds_k-c)
\end{align*}
which, combined with definition of $\tx_k$ in \eqref{xtilde}, yields \eqref{aux.3}.

\end{proof}

In order to show that the proximal ADMM is an instance of the NE-HPE framework,  we need to introduce the elements required
by the setting of Section \ref{sec:smhpe}, namely, the space $\mathcal{Z}$, the seminorm $\|\cdot\|$ on $\Z$,  the  distance generating function $w: \Z \to [-\infty,\infty]$ and
the convex set $Z \subset \mbox{int} (\dom w)$. We consider
$\mathcal{Z}:=\mathcal{S}\times \Y\times\X$ and endow it with the inner product given by 
\[
\inner{z}{z'} := \langle s,s'\rangle_\mathcal{S}+ \langle y,y'\rangle_\Y+\langle x,x'\rangle_\X
\quad \forall  \,z=(s,y,x), \, z'=(s,y,x).
\]
The seminorm $\|\cdot\|$, the function $w$ and the set $Z$ are defined as 
\begin{equation}\label{df:norm_admm}
 \|z\|:=\left(\|s\|_{\mathcal{S},H}^2+\|y\|_{\Y,(G+\beta C^*C)}^2+\frac{1}{\beta\theta}\|x\|_\X^2\right )^{1/2} \quad w(z):=\frac{1}{2}\|(s,y,x)\|^2, \quad Z:=\Z
\end{equation}
for every $z=(s,y,x) \in \Z$.
Clearly, the Bregman distance associated with $w$ is given by
\begin{equation}\label{df:dw_admm}
(dw)_z(z')=\frac{1}{2}\|s'-s\|_{\mathcal{S},H}^2+\frac{1}{2}\|y'-y\|^2_{\Y,(G+\beta C^*C)} +\frac{1}{2\beta\theta}\|x'-x\|_\X^2
\end{equation}
for every $z=(s,y,x) \in \Z$ and $z'=(s',y',x') \in \mathcal{Z}$.
 
 Using Proposition \ref{propdualnorm} and the fact that  $\|\cdot\|= \|\cdot\|_{\Z,Q}$ where $Q$ is the self-adjoint positive semidefinite linear operator given by
\[
Q(s,y,x)=(Hs,(G+\beta C^*C) y,x/(\beta\theta)) \quad \forall (s,y,x)\in \Z,
\]
it is easy to see that the function $w$ is a $(1,1)$-regular distance generating function with respect to $(Z,\|\cdot\|)$.

To simplify some relations in the proofs given below, define
\begin{equation}\label{def:deltasyx}
\Delta s_k = s_{k}-s_{k-1}, \quad  \Delta y_k=y_{k}-y_{k-1} , \quad  \Delta x_k=x_{k}-x_{k-1}.
\end{equation}
The following technical result will be used to prove that  the proximal ADMM is an instance of the NE-HPE framework.

\begin{lemma}\label{lem:deltak} Let  $\{(s_k,y_k,x_k)\}$ be the sequence generated by the proximal ADMM. Then, the following statements hold:
\begin{itemize}
\item [(a)]   if $\theta <2$, then
$$
\frac{1}{\sqrt{\theta}}\left(\frac{1}{2}\|\Delta y_1\|_{\Y,G}^2- \frac{1}{\sqrt{\theta}}\langle C\Delta y_1,\Delta x_1 \rangle_\X\right) \leq  \tau_\theta  d_0
$$ 
where $d_0$ and $\tau_\theta$ are as in \eqref{def:d0admm} and \eqref{def:tau}, respectively.

\item [(b)] for any  $\theta>0$, we have
\begin{equation*}
\frac{1}{\theta} \langle C\Delta y_k,\Delta x_k \rangle_\X
\geq \frac{1-\theta}{\theta} \langle C\Delta y_k,\Delta x_{k-1} \rangle_\X
+\frac{1}{2}\|\Delta y_k\|_{\Y,G}^2-\frac{1}{2}\|\Delta y_{k-1}\|_{\Y,G}^2,\quad \forall k\geq 2.
\end{equation*}

\end{itemize}
\end{lemma}
\begin{proof}
(a) Let  a point  $z^*:=(s^*,y^*,x^*)$ be  such that $0\in T(s^*,y^*,x^*)$ (see assumption {\bf A1}).  Since   $\langle x,x' \rangle_\X \leq (1/2)(\|x\|_{\X}^2+\|x'\|_{\X}^2) $ for every $x,x' \in \X$,  using \eqref{def:deltasyx}  we obtain
\begin{align*}\label{eq_00000001}
\frac{1}{2}\|\Delta y_1\|_{\Y,G}^2- \frac{1}{\sqrt{\theta}}\langle C\Delta y_1,\Delta x_1 \rangle_\X
& \leq \frac{1}{2\beta\theta}\|\Delta x_1\|_\X^2+\frac{\beta}{2}\|C\Delta y_1\|_\X^2+\frac{1}{2}\|\Delta y_1\|_{\Y,G}^2\\
 &\leq \frac{1}{\beta\theta}{\|x_1-x^*\|_\X^2+{\beta} \|C(y_1-y^*)\|_\X^2}+\|y_1-y^*\|_{\Y,G}^2\\ &+\frac{1}{\beta\theta}{\|x_0-x^*\|_\X^2+ {\beta}\|C(y_0-y^*)\|_\X^2}+\|y_0-y^*\|_{\Y,G}^2
\end{align*}
which, combined with \eqref{df:dw_admm} and simple calculus, yields
\begin{equation}\label{eq_000000012}
\frac{1}{\sqrt{\theta}}\left(\frac{1}{2}\|\Delta y_1\|_{\Y,G}^2- \frac{1}{\sqrt{\theta}}\langle C\Delta y_1,\Delta x_1 \rangle_\X\right)\leq  \frac{2}{\sqrt{\theta}}\left((dw)_{z_1}(z^*)+(dw)_{z_0}(z^*)\right).
\end{equation}
On the other hand, consider  
\begin{equation}\label{def:zk1}
  z_{0}=(s_{0},y_{0},x_{0}), \quad z_{1}=(s_{1},y_{1},x_{1}), \quad \tilde{z}_1=(s_1,y_1,\tilde{x}_1), \quad \lambda_1=1, \quad \varepsilon_1=0. 
 \end{equation}
Lemma~\ref{pr:aux} implies that  inclusion \eqref{breg-subpro} is satisfied for $(z_0, z_1,\tilde{z}_1,\lambda_1,\varepsilon_1)$ with $T$  and   $dw$ as in \eqref{FAB} and \eqref{df:dw_admm}, respectively. Hence, 
 it follows from Lemma~\ref{lema_desigualdades}(a) with $z=z^*$, $\lambda_1=1$ and  the fact that $\langle r_1, \tilde{z}_1 - z^*\rangle\geq 0$ (because $0 \in T(z^*)$ and $r_1 \in T(\tz_1)$)
 that
\begin{equation}\label{ineq_lema007d1d0}
(dw)_{z_1}(z^*)\leq (dw)_{z_0}(z^*) +(dw)_{z_1}(\tilde{z}_1)-(dw)_{z_0}(\tilde{z}_1).
\end{equation}
Using the definitions in  \eqref{df:dw_admm} and \eqref{def:zk1}   and equation in 
 \eqref{aux.3}, we obtain 
\begin{align*}
(dw)_{z_1}(\tilde{z}_1)-(dw)_{z_0}(\tilde{z}_1)&\leq \frac{1}{2\beta\theta}\|\tx_1-x_1\|_\X^2-\frac{\beta}{2}\|C(y_1-y_0)\|_\X^2-\frac{1}{2\beta\theta}\|\tx_1-x_0\|_\X^2 \\
&= \frac{(\theta-1)}{2\beta\theta^2}\|x_1-x_0\|_\X^2-\frac{1}{2}\left\| \frac{x_1-x_0} {\theta\sqrt{\beta}}+\sqrt{\beta}{C(y_1-y_0)}\right\|_\X^2 \\
&\leq \frac{(\theta-1)}{2\beta\theta^2}\|x_1-x_0\|_\X^2.
\end{align*}
If $\theta\in (0,1]$, then the last inequality implies that 
\begin{equation}\label{ineq_s23}
(dw)_{z_1}(\tilde{z}_1)\leq (dw)_{z_0}(\tilde{z}_1).
\end{equation}
 Now, if $\theta\in (1,2)$, we have
\begin{align*}
(dw)_{z_1}(\tilde{z}_1)-(dw)_{z_0}(\tilde{z}_1)&\leq \frac{(\theta-1)}{2\beta\theta^2}\|x_1-x_0\|_\X^2\leq \frac{2(\theta-1)}{\theta}\left( \frac{\|x_1-x^*\|_\X^2}{2\beta\theta}+\frac{\|x_0-x^*\|_\X^2}{2\beta\theta}\right)\\
&\leq \frac{2(\theta-1)}{\theta} \left[(dw)_{z_1}(z^*)+(dw)_{z_0}(z^*)\right]
\end{align*}
where   the second inequality is due to  the fact that $2ab\leq a^2+b^2$ for all $ a,b\geq 0$, and the last inequality is due to \eqref{df:dw_admm} and definitions of $z_0,z_1$ and $z^*$.
Hence, combining the last estimative with \eqref{ineq_lema007d1d0},   we obtain 
$$
(dw)_{z_1}(z^*)\leq \frac{\theta}{2-\theta}\left(1+\frac{2(\theta-1)}{\theta}\right)(dw)_{z_0}(z^*)=\frac{3\theta-2}{2-\theta}(dw)_{z_0}(z^*),
$$ 
which, combined  with \eqref{ineq_s23},  yields
\[
(dw)_{z_1}(z^*)\leq \max\left\{1,\frac{3\theta-2}{2-\theta}\right\}(dw)_{z_0}(z^*).
\]
Therefore, statement (a) follows from \eqref{eq_000000012}, the last inequality, definition of $\tau_\theta$ in \eqref{def:tau} and   the fact that $d_0$ (as defined in  \eqref{def:d0admm}) satisfies $d_0= \inf_{z \in T^{-1}(0)} (dw)_{z_0}(z)$.
 
(b) From the inclusion  \eqref{aux.2} and relation \eqref{aux.3},  we see that, for every $j\geq1$,
$$
 \partial f(y_j) \ni C^*(\tilde{x}_j-\beta C(y_j-y_{j-1}))-G(y_j-y_{j-1})= \frac{1}{\theta} C^*(x_j-(1-\theta)x_{j-1})-G(y_j-y_{j-1}).
$$
For every $k\geq 2$, using the previous inclusion for $j=k-1$ and $j=k$, it follows from the monotonicity of the subdifferential of $f$ that
\begin{align*}
0&\leq \left \langle \frac{1}{\theta} C^*(x_k-x_{k-1})- \frac{(1-\theta)}{\theta} C^*(x_{k-1}-x_{k-2})- G(y_k-y_{k-1})+ G(y_{k-1}-y_{k-2}),y_k-y_{k-1}\right\rangle_\Y,
\end{align*}
which, combined with \eqref{def:deltasyx}, yields
\[
\frac{1}{\theta} \langle C\Delta y_k,\Delta x_k \rangle_\X
\geq \frac{(1-\theta)}{\theta} \langle C\Delta y_k,\Delta x_{k-1} \rangle_\X+\|\Delta y_k\|_{\Y,G}^2- \langle G \Delta y_{k-1},\Delta y_k\rangle_\Y.
\]
Hence, item~(b)  follows from the last inequality and the fact that $\langle Gy,y'\rangle_\Y \leq (1/2)(\|y\|_{\Y,G}^2+\|y'\|_{\Y,G}^2) $ for every $y,y' \in \Y$. Therefore, the proof of the lemma is concluded.
\end{proof}
We now present  some properties of the parameter $\sigma_\theta$ defined in~\eqref{def:sigma}.  

\begin{lemma}\label{pro:sigma} 
Let $\theta \in (0,(\sqrt{5}+1)/2]$ be given and consider  the parameter $\sigma_\theta$ as defined in \eqref{def:sigma}. Then,  the  following statements hold:
\begin{itemize}
\item[(a)]
$\sigma=\sigma_\theta$   is the largest root of the equation
$\det(M_{\theta}(\sigma))=0$ and
 $\det(M_{\theta}(\sigma)) > 0$ for every $\sigma>\sigma_\theta$ where $\det(\cdot)$ denotes the determinant function and
\begin{equation} \label{matrixtheta>1}
M_{\theta}(\sigma) := \left[
\begin{array}{cc} 
 \sigma(1+\theta)-1& (\sigma+\theta-1)(1-\theta)\\[2mm]
(\sigma+\theta-1)(1-\theta)&  \sigma-(1-\theta)^2\\
\end{array} \right];
\end{equation}
\item[(b)] $1/3<\max\{(1-\theta)^2,1-\theta, {1}/{(1+\theta)}\}\leq \sigma_\theta \leq 1$;
\item[(c)] the matrix $M_{\theta}(\sigma)$ in \eqref{matrixtheta>1}  is  positive semidefinite for $\sigma=\sigma_\theta$.
\end{itemize}
\end{lemma}
\begin{proof} (a) It is a simple algebraic computation to see that $\sigma=\sigma_\theta$ is the largest root of the second-order equation
$\det(M_{\theta}(\sigma))=0$.

(b) The second inequality  follows by  (a) and the fact that  $\det(M_\theta(\sigma))\leq 0$  for  $\sigma$ equal to   $(1-\theta)^2$, $1-\theta$ and ${1}/{(1+\theta)}$.
Now, the first and third inequalities are due to the fact that $\theta\in (0,(\sqrt{5}+1)/2]$ and $1/3 \leq {1}/(1+\theta)$.

(c) Statements (a) and (b) imply that $\det(M_{\theta}(\sigma_\theta))=0$ and  the main diagonal entries of  $M_{\theta}(\sigma_\theta)$
are nonnegative.  Since $M_{\theta}(\sigma)$ is symmetric, we then conclude that  (c)  holds.
\end{proof}

 The next result shows that the proximal ADMM can be seen as an instance of the NE-HPE framework. 

\begin{theorem}\label{th:admm_hpe} Consider
the operator $T$  and  Bregman distance $dw$ as in \eqref{FAB} and \eqref{df:dw_admm}, respectively. Let $\{(s_k,y_k,x_k)\}$  be the sequence generated by the proximal ADMM with  $\theta\in (0,(\sqrt{5}+1)/2]$ and consider $\{\tx_k\}$ as  in \eqref{xtilde}. Define
\begin{equation}\label{def:zk}
 \quad z_{k-1}=(s_{k-1},y_{k-1},x_{k-1}), \quad \tilde{z}_k=(s_k,y_k,\tilde{x}_k), \quad \lambda_k=1, \quad \varepsilon_k=0 \quad \forall k\geq 1,
\end{equation}
and the sequence $\{\eta_k\}$ as 
\begin{equation}\label{eta}
 \eta_0={\tau_\theta d_0}, \quad \eta_k= {[\sigma_\theta-(\theta-1)^2]}\frac{\|\Delta x_k\|_\X^2}{2\beta\theta^3}+
\frac{\sigma_\theta+\theta-1}{2\theta}\|\Delta y_k\|_{\Y,G}^2, \quad  \forall k\geq 1
 \end{equation}
 where  $d_0$, $\sigma_\theta$, $\tau_\theta$ and $(\Delta x_k,\Delta y_k)$   are  as in  \eqref{def:d0admm},  \eqref{def:sigma}, \eqref{def:tau} and \eqref{def:deltasyx}, respectively. 
Then, the sequence $\{(z_k,\tilde{z}_k,\lambda_k,\varepsilon_k,\eta_k)\}$ is an instance of   the NE-HPE framework   with input $z_0=(s_0,y_0,x_0), \eta_0$ and $\sigma=\sigma_\theta$.
\end{theorem}
\begin{proof}
The inclusion \eqref{breg-subpro} follows from  \eqref{aux.0}-\eqref{aux.1}, \eqref{def:zk} and  definitions of $T$ and $dw$.
Now it remains to show that the error condition \eqref{breg-cond1} holds. First of all, it follows from \eqref{aux.3}, \eqref{df:dw_admm}, \eqref{def:deltasyx} and \eqref{def:zk}  that  
\begin{align}
(dw)_{ z_{k}}(\tilde{z}_k)  +\lambda_k\varepsilon_k&=\frac{1}{2\beta\theta}\|\tx_k-x_k\|_\X^2= \frac{1}{2\beta\theta}\left\|\beta C\Delta y_k +\frac{1-\theta}{\theta}{\Delta x_k}\right\|_\X^2  \nonumber \\
                                                                              & = \frac{\beta}{2\theta}\|C\Delta y_k\|_\X^2 +\frac{(1-\theta)}{ \theta^2}\inner { C\Delta y_k}{\Delta x_k}_\X+\frac{(1-\theta)^2}{2\beta\theta^3}\|\Delta x_k\|_\X^2 \label{aux.102}.
\end{align}    
Also, \eqref{df:dw_admm}, \eqref{def:deltasyx} and  \eqref{def:zk}  imply that
    \begin{equation}\label{ad90}
     (dw)_{ z_{k-1}}(\tilde{z}_k)     = \frac{1}{2}\|\Delta s_k\|_{\mathcal{S},H}^2+\frac{1}{2}\|\Delta y_k\|_{\Y,G}^2+ \frac{\beta}{2} \|C\Delta y_k\|_\X^2+\frac{1}{2\beta\theta}\|x_{k-1}-\tilde x_k\|_\X^2.
   \end{equation}        
It follows from \eqref{aux.3} and \eqref{def:deltasyx} that  
   \begin{equation*}\label{eq:righpecond_eta}
\|x_{k-1}-\tilde x_k\|_\X^2 =\left\|\beta C\Delta y_k +\frac{1}{\theta}\Delta x_k\right\|_\X^2 =
                                           {\beta^2}\|C\Delta y_k\|_\X^2
                                                                                          +\frac{2\beta}{\theta}\inner { C\Delta y_k}{\Delta x_k}_\X+\frac{1}{\theta^2}\|\Delta x_k\|_\X^2                                                                                                                                                                                                                                                                  
\end{equation*} 
which, combined with \eqref{ad90}, yields 
\begin{equation}\label{aux.103}
     (dw)_{ z_{k-1}}(\tilde{z}_k)     = \frac{\|\Delta s_k\|_{\mathcal{S},H}^2}{2}+\frac{\|\Delta y_k\|_{\Y,G}^2}{2}+ \frac{\beta(\theta+1)\|C\Delta y_k\|_\X^2}{2\theta} +\frac{1}{\theta^2}\inner { C\Delta y_k}{\Delta x_k}_\X+\frac{\|\Delta x_k\|_\X^2}{2\beta\theta^3}.
 \end{equation}
Therefore, combining \eqref{aux.102} and \eqref{aux.103}, we see, after simple algebraic manipulations, that the error condition  \eqref{breg-cond1} is satisfied if and only if  
\begin{align}
[\sigma(1+\theta)-1 ]\frac{\beta\|C\Delta y_k\|_\X^2}{2\theta}&+\left[\sigma-(\theta-1)^2\right]\frac{\|\Delta x_k\|_\X^2}{2\beta\theta^3}  +\sigma \frac{\|\Delta y_k\|_{\Y,G}^2}{2}+\frac{(\sigma +\theta-1)}{\theta^2}\langle C\Delta y_k, \Delta x_k\rangle_\X\nonumber\\
  & \geq   \eta_k-\eta_{k-1}-\sigma \frac{\|\Delta s_k\|_{\mathcal{S},H}^2}{2}\label{ineq:proximal ADMMhpe}.
 \end{align}
We now show that inequality \eqref{ineq:proximal ADMMhpe} with $\sigma=\sigma_\theta$ holds for $k=1$.
Indeed, it follows from definition of  $\eta_1$  and $\sigma_\theta \geq {1}/{(1+\theta)}$ (see Lemma~\ref{pro:sigma}~(b)) that 
\begin{align*}
&[\sigma_\theta(1+\theta)-1 ] \frac{\beta\|C\Delta y_1\|_\X^2}{2\theta}+
\left[\sigma_\theta-(1-\theta)^2\right]\frac{\|\Delta x_1\|_\X^2}{2\beta\theta^3}  +
\sigma_\theta \frac{\|\Delta y_1\|_{\Y,G}^2}{2} 
+\frac{(\sigma_\theta +\theta-1)}{\theta^2}\langle C\Delta y_1,\Delta x_1 \rangle_\X \\
&\geq 
\left[\sigma_\theta-\frac{\sigma_\theta+\theta-1}{\theta}  +\frac{\sigma_\theta +\theta-1}{\theta^{3/2}} \right] \frac{\|\Delta y_1\|_{\Y,G}^2}{2}+ \eta_1  +\frac{(\sigma_\theta +\theta-1)}{\theta^{3/2}}\left(  \frac{1}{\sqrt{\theta}}\langle C\Delta y_1,\Delta x_1 \rangle_\X -\frac{1}{2}\|\Delta y_1\|_{\Y,G}^2\right)\\
&\geq \left[\sigma_\theta-\frac{\sigma_\theta+\theta-1}{3\theta}   \right] \frac{\|\Delta y_1\|_{\Y,G}^2}{2}+\eta_1-\frac{(\sigma_\theta +\theta-1)}{\theta}{\tau_\theta d_0}\\
&  \geq     \eta_1  - \eta_0
\end{align*}
where the second inequality follows from the fact that $\sqrt{\theta}\leq 3/2$ and Lemmas~\ref{lem:deltak}(a) and  \ref{pro:sigma}(b),  and   the 
third inequality is due to the fact that  $1/3\leq\sigma_\theta \leq 1$ (see Lemma~\ref{pro:sigma}(b))  and definition of $\eta_0$.
Therefore,  inequality  \eqref{ineq:proximal ADMMhpe} holds with $k=1$ and $\sigma=\sigma_\theta$.

We next show that inequality  \eqref{ineq:proximal ADMMhpe} with $\sigma=\sigma_\theta$ holds for $k \ge 2$.
Using Lemma~\ref{lem:deltak}(b) and the definition of $\{\eta_k\}$ in \eqref{eta}, we see, after simple calculus, 
that a sufficient condition for \eqref{ineq:proximal ADMMhpe} to hold with $\sigma=\sigma_\theta$ and $k \ge 2$ is that
\[
 (\sigma_\theta(1+\theta)-1)\beta\frac{\|C\Delta y_k\|_\X^2}{2}+[\sigma_\theta-(1-\theta)^2]\frac{\|\Delta x_{k-1}\|_\X^2}{2\beta\theta^2} 
  +\frac{(\sigma_\theta + \theta-1)(1-\theta)}{\theta}\langle C\Delta y_k ,\Delta x_{k-1}\rangle_\X\geq 0.
\]
Hence, since the last inequality holds due to Lemma~\ref{pro:sigma}(c),  
we conclude  the proof of the theorem.
\end{proof}

Now we are ready to present the proof of the pointwise convergence rate  of the proximal ADMM.

\noindent {\bf Proof of Theorem~\ref{th:pointwise}:} 
Since $\sigma_\theta \in [0,1)$ for any  $\theta\in (0,(\sqrt{5}+1)/2)$ and  $w$ as defined in \eqref{df:norm_admm} is a $(1,1)$-regular distance generating function 
and  $\underline{\lambda}:= \inf \lambda_k=1$,
we obtain  by combining Theorems~\ref{th:admm_hpe} and  \ref{th:pointwiseHPE}(a) that inclusion \eqref{eq:th_incADMMtheta1}  holds and there exists $i\leq k$ such that 
 \[
 \left(\|s_{i-1}-s_i\|_{\mathcal{S},H}^2+\|y_{i-1}-y_i\|_{\Y,(G+\beta C^*C)}^2+\frac{1}{\beta\theta}\norm{x_{i-1}-x_i}_\mathcal{X}^2\right)^{1/2}\leq
\frac{2}{\sqrt{k}} \sqrt{\frac{(1+\sigma_\theta)d_0+2\eta_0}{1-\sigma_\theta}},
\]
where we also used the definition of the norm $\|\cdot\|$ in \eqref{df:norm_admm} and Proposition~\ref{propdualnorm}.
The result now follows  from the last inequality and 
definition of  $\eta_0$ in  \eqref{eta}.  \hfill{ $\square$}

 In order to establish the ergodic convergence rate of the proximal ADMM, we need the next auxiliary result.
 
 \begin{lemma}\label{lem:aux3}
Let $\{(s_k,y_k,x_k)\}$  be the sequence generated by the proximal ADMM and $\{\tx_k\}$ be  given by  \eqref{xtilde}.
Then, the pair $(z_{k-1},\tz_k)$ as defined in \eqref{def:zk} satisfies
\[
(dw)_{ z_{k-1}}(\tilde{z}_k)  \leq \frac{4(1+\tau_\theta)(\theta^2+\theta+1)}{\theta^2}  d_0 \quad k\geq1,
\]
where   $dw$ is  the Bregman distance given in \eqref{df:dw_admm}, $d_0$ and $\tau_\theta$ are  as in  \eqref{def:d0admm} and \eqref{def:tau}, respectively.
\end{lemma}
 \begin{proof} It follows from \eqref{df:dw_admm} and \eqref{def:zk} that 
   \begin{equation}\label{eq:pij}
 (dw)_{ z_{k-1}}(\tilde{z}_k)  = \frac{1}{2}\|\Delta s_k\|_{\mathcal{S},H}^2+\frac{1}{2}\|\Delta y_k\|_{\Y,G}^2+ \frac{\beta}{2} \|C\Delta y_k\|_\X^2+\frac{1}{2\beta\theta}\|x_{k-1}-\tilde x_k\|_\X^2.
  \end{equation}                                                                                      
  On the other hand, using \eqref{aux.3} we have       
   \begin{align*}                                                                                        
\frac{1}{2\beta\theta}\|x_{k-1}-\tilde x_k\|_\X^2&=                                                        
                                                                                          \frac{1}{2\beta\theta}\|\beta C\Delta y_k +\frac{\Delta x_k}{\theta}\|_\X^2 \\
                                                                                          &= \frac{\beta}{2\theta} \|C\Delta y_k\|_\X^2 
                                                                                          +\frac{1}{\beta \theta^2}\inner {\beta C\Delta y_k}{\Delta x_k}_\X+\frac{1}{2\beta\theta^3}\|\Delta x_k\|_\X^2 \\
                                                                                                                                                                                       &\leq \frac{\beta(\theta+1)}{2\theta^2} \|C\Delta y_k\|_\X^2 +\frac{\theta+1}{2\beta\theta^3}\|\Delta x_k\|_\X^2 
\end{align*}                                                                              
where the inequality is due to  Cauchy-Schwarz inequality and  the fact that $2ab\leq a^2+b^2$ for all $ a,b\geq 0$.
Combining the last inequality and~\eqref{eq:pij}, we have                                                                                         
     \begin{align*}                                                                 
  (dw)_{ z_{k-1}}(\tilde{z}_k)& \leq\frac{(\theta^2+\theta+1)}{\theta^2} \Big[\frac{1}{2}\|\Delta s_k\|_{\mathcal{S},H}^2+
  \frac{1}{2}\|\Delta y_k\|_{\Y,G}^2+ \frac{\beta}{2} \|C\Delta y_k\|_\X^2+  \frac{1}{2\beta\theta}\|\Delta x_k\|_\X^2\Big] \\
                                           &\leq \frac{2(\theta^2+\theta+1)}{\theta^2} \left[ (dw)_{ z_{k-1}}(z^*)+  (dw)_{ z_{k}}({z}^*)  \right] 
 \end{align*}   
where the last inequality is due to   \eqref{df:dw_admm} and the fact that $2ab\leq a^2+b^2$ for all $ a,b\geq 0$.
Hence, since by Theorem~\ref{th:admm_hpe} the proximal ADMM is an instance of the NE-HPE framework,  it  follows from the last estimative and  Lemma~\ref{lema_desigualdades}(d)  that 
 \[  (dw)_{ z_{k-1}}(\tilde{z}_k)     \leq \frac{4(\theta^2+\theta+1)}{\theta^2}  ((dw)_{ z_{0}}(z^*)+\eta_0)\]
which, combined with the definition of $\eta_0$ in \eqref{eta} and the fact that $d_0$ (as defined in  \eqref{def:d0admm}) satisfies $d_0= \inf_{z \in T^{-1}(0)} (dw)_{z_0}(z)$, proves the result.
\end{proof}
 
Next we present the proof of the ergodic iteration-complexity bound for the proximal ADMM.

\noindent {\bf Proof of Theorem~\ref{th:ergodicproximal ADMM}:} First, it follows from  Theorem~\ref{th:admm_hpe}  that  the proximal ADMM with $\theta\in (0,(\sqrt{5}+1)/2]$ is an instance of the NE-HPE  applied to problem \eqref{FAB} in which $\sigma:=\sigma_\theta$, $\{(z_k,\tilde{z}_k,\lambda_k,\varepsilon_k)\}$ and $\{\eta_k\}$ are as defined in  \eqref{def:sigma}, \eqref{def:zk} and \eqref{eta}, respectively.
Moreover, inclusions \eqref{aux.0}-\eqref{aux.1} is equivalent to 
\[ 
H(s_{k-1}-s_{k})+D^*\tilde{x}_k \in \partial g(s_k),\quad (G+\beta C^*C)(y_{k-1}-y_{k})+C^*\tilde{x}_k \in \partial f(y_k), \quad \frac{1}{\beta\theta}(x_{k-1}-x_k)=Cy_k+Ds_k-c.
\]
Hence, using \eqref{eq:jase12} we obtain trivially the third inclusion of \eqref{eq:th_incADMMtheta<1} while the first and second inclusions of \eqref{eq:th_incADMMtheta<1}  follow from \eqref{eq:jase12} and the transportation formula (see  \cite[Theorem~2.3]{Bu-Sag-Sv:teps1}).
   Now, since  $w$ as defined in \eqref{df:norm_admm} is a $(1,1)$-regular distance generating function 
and  $1=\underline{\lambda}= \inf \lambda_k$,
we obtain from Theorem~\ref{th:ergHPE} and Lemma~\ref{lem:aux3} that
\begin{align}\label{eq:al90}
 \|(r^a_{k,s}, r^a_{k,y}, r^a_{k,x})\|^* &\le \frac{2\sqrt{2(d_0+\eta_0)}}{k}, \,
 \varepsilon^a_{k} \leq 
\frac{ 3[  3\theta^2(d_{0} +\eta_0)+4\sigma_{\theta}(1+\tau_{\theta})(\theta^2+\theta+1)d_0]  }{\theta^2k}, \forall k\geq 1.
\end{align}
Hence, since  $(r^a_{k,s}, r^a_{k,y}, r^a_{k,x})=(H (s_{k-1}^a-s_k^a),  
(G+\beta C^*C)(y_{k-1}^a-y_k^a),(x_{k-1}^a-x_k^a)/(\beta\theta)$, it follows from  the definition of $\|\cdot\|$ in \eqref{df:norm_admm}  and  Proposition~\ref{propdualnorm} that
\begin{align*}
 \|(r^a_{k,s}, r^a_{k,y}, r^a_{k,x})\|^* = \left(\|s_{k-1}^a-s_k^a\|_{\mathcal{S},H}^2+\|y_{k-1}^a-y_k^a\|_{\Y,(G+\beta C^*C)}^2+\frac{1}{\beta\theta}\norm{x_{k-1}^a-x_k^a}_\mathcal{X}^2\right)^{1/2}.
\end{align*}
Therefore, the result follows  from  \eqref{eq:al90} and
the definition of  $\eta_0$ in  \eqref{eta}. \hfill{ $\square$}

\def\cprime{$'$}


\end{document}